\documentclass[graybox]{svmult}

\usepackage{mathptmx}       
\usepackage{helvet}         
\usepackage{courier}        
\usepackage{type1cm}        
\usepackage{makeidx}         
\usepackage{graphicx}        
\usepackage{multicol}        
\usepackage[bottom]{footmisc}

\usepackage{amsmath, amssymb}
\usepackage{txfonts}
\AtBeginDocument{\mathcode`v=\varv}

\newtheorem{Th}{Theorem}[section] 
\newtheorem{Lem}[Th]{Lemma} 
\newtheorem{Prop}[Th]{Proposition} 
 
\newtheorem{Def}[Th]{Definition} 
 
\newtheorem{Rem}[Th]{Remark}
\newtheorem{Alg}[Th]{Algorithm}

\makeindex             

\def\R{{\mathbb R}}

\def\C{{\mathcal C}}

\def\H{{\mathcal H}}

\newcommand{\argmin}{\mathop{\rm arg~min}\limits}

\begin{document}

\title*{Deformation problem for glued elastic bodies
  and an alternative iteration method} 
\author{Masato Kimura and Atsushi Suzuki}
\institute{
Masato Kimura \at Faculty of Mathematics and Physics, Kanazawa University,
Kakuma, Kanazawa, 920-1192, Japan, \email{mkimura@se.kanazawa-u.ac.jp}
\and 
Atsushi Suzuki \at Cybermedia Center, Osaka University,
  Machikaneyama, Toyonaka, Osaka, 560-0043, Japan,
  \email{atsushi.suzuki@cas.cmc.osaka-u.ac.jp}
}

%
%
\maketitle

\abstract{
  We study a mathematical model for deformation of glued elastic bodies in 2D or 3D,
  which is a linear elasticity system with adhesive force on the glued surface.
  We reveal a variational structure of the model
  and prove the unique existence of a weak solution based on it.
  Furthermore, we also consider an alternating iteration method
  and show that it is nothing but an alternating minimizing method
  of the total energy.
  The convergence of a monolithic formulation and the
  alternating iteration method are numerically studied with the finite element method.}

\section{Introduction}
\label{sec:1}
We consider a mathematical model which describes
deformation of two elastic bodies glued to each other
on a surface. 
The understanding of such glued structure 
or adhesive bonding process is important in industrial and scientific applications,
especially in the case that
the mechanical bonding
technique exhibits
its disadvantages comparing with the adhesive one, e.g.
bonding between soft materials, or one between very small scaled materials.
The importance of mathematical modeling
and numerical simulation is increasing
in the design of desirable mechanical properties of
composite materials with glued layer structure.

In mathematics, M. Fr\'emond \cite{Fremond2002}
and T. Roub\'i\v{c}ek {\it et al.} \cite{RSZ2009}
proposed mathematical models of such glued structure
and its delamination process.
R. Scala \cite{Scala2017} studied more extended
delamination model
with kinetic and viscoelastic terms
and proved existence of a solution.
For further mathematical studies on the delamination process,
we refer the above works and references therein.

In this paper, we concentrate on the stationary
deformation problem of the glued structure,
which is a linear elasticity system with adhesive force on the glued surface.
A similar stationary problem also appears in
the implicit time discretization of the above delamination models
\cite{Yoneda2018}.

The outline of this paper is as follows.
In Section~\ref{sec:2}, we describe the deformation model
and give a definition of a weak solution.
Section~\ref{sec:3} is devoted to
review several known consequences from the coercivity
of a bilinear form; the existence and the uniqueness
of a weak solution, a variational principle,
and an error estimate of a finite element approximation.

For the purpose of efficient numerical computation of
the obtained weak form of our model in 2D and 3D,
we propose an alternating iteration method in Section~\ref{sec:4}.
The alternating iteration method
was proposed in \cite{Yoneda2018}
and was used to simulate the vibration-delamination model proposed in \cite{Scala2017}.
We will show that it is nothing but
an alternating energy minimization procedure.
In particular, it generates
a sequence of displacements which monotonically
decreases the total energy (Theorem~\ref{GS-energy-decay}).

In Section~\ref{sec:5}, we consider finite element discretization.
We give discrete forms of the monolithic method and the alternating iteration method,
and prove that the sequence generated by the alternating iteration method
converges to the discrete solution by the monolithic method. 
Those theoretical results is also supported by
numerical experiments in three dimensional setting. 

\section{Deformation of glued elastic bodies}
\label{sec:2}
We consider a bounded Lipschitz domain $\Omega\subset \R^d,~(d=2,3)$.
We suppose that
\begin{equation*}
  \Omega\setminus\Gamma=\Omega_1\cup\Omega_2,\quad
  \Omega_1\cap\Omega_2 =\emptyset,
\end{equation*}
where $\Gamma$ is a Lipschitz surface which is the common boundary of two disjoint Lipschitz domains
$\Omega_1$ and $\Omega_2$ as shown in Fig.~\ref{ex1} and Fig.~\ref{ex2}.
We denote by $\nu$ the unit normal vector on $\Gamma$ pointing from $\Omega_1$ into $\Omega_2$,
and the one on $\partial\Omega$ pointing outward.
We suppose that the boundary $\partial \Omega$ is decomposed to the following
disjoint portions:
\begin{equation*}
  \partial\Omega=\partial_D\Omega\cup\overline{\partial_N\Omega},~~
  \partial_D\Omega\cap\overline{\partial_N \Omega}=\emptyset,~~
 \H^{d-1}(\overline{\partial_N \Omega}\setminus \partial_N \Omega)=0,
\end{equation*}
where $\partial_D\Omega$ and $\partial_N\Omega$ are relatively open subsets of $\partial\Omega$
and $\H^{d-1}$ denotes the $d-1$ dimensional Hausdorff measure.
We also define $\partial_D\Omega_i:=\partial_D\Omega\cap\partial\Omega_i$
and 
$\partial_N\Omega_i:=\partial_N\Omega\cap\partial\Omega_i$ for $i=1,2$.
We assume that $\partial_D\Omega$ is a nonempty relatively open subset of $\partial \Omega$, and 
$\H^{d-1}(\partial_D\Omega_1)$ or $\H^{d-1}(\partial_D\Omega_2)$ is positive. 
Without loss of generality, we always suppose $\H^{d-1}(\partial_D\Omega_1)>0$ 
 (Fig.~\ref{ex1}, Fig.~\ref{ex2}).

\begin{figure}[h]
  \begin{minipage}{0.49\textwidth}
    \begin{center}
      \includegraphics[bb= 0 650 250 850,clip,width=1.1\textwidth]{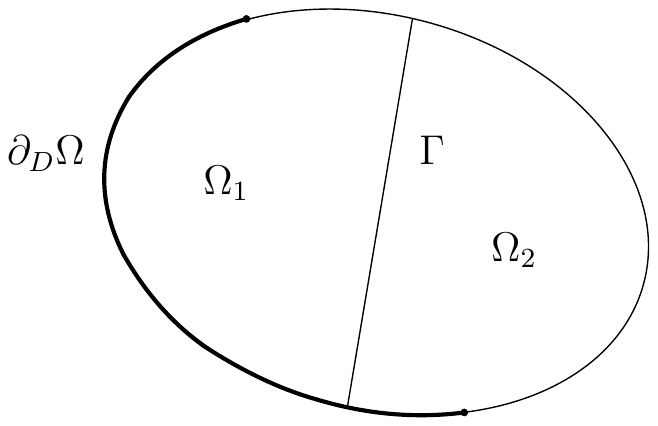}
    \end{center}
    \caption{$\partial_D\Omega\subset\partial\Omega_1\cup\partial\Omega_2$}
    \label{ex1}
  \end{minipage}
    \begin{minipage}{0.49\textwidth}
    \begin{center}
      \includegraphics[bb= 0 650 250 850,clip,width=1.1\textwidth]{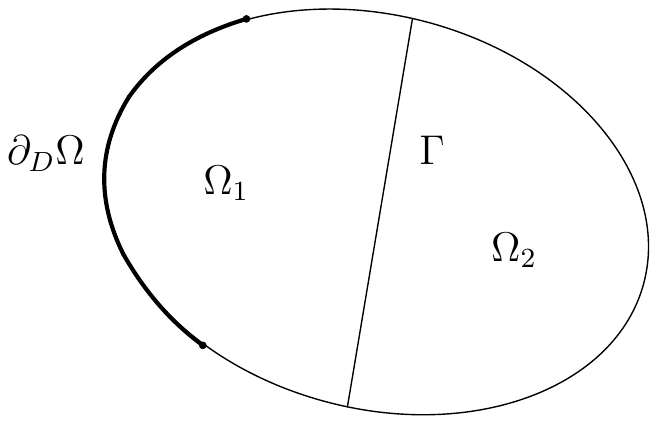}
    \end{center}
    \caption{$\partial_D\Omega\subset\partial\Omega_1$}
    \label{ex2}
  \end{minipage}
\end{figure}

In this paper, we consider the following stationary deformation model of two elastic bodies $\Omega_1$ and $\Omega_2$
which are glued by an adhesive on $\Gamma$.
We consider an adhesion force but ignore a friction force on the interface.

The problem is to find $u:\Omega\setminus\Gamma\to\R^d$ such that:
\begin{equation}\label{model}
  \left\{
  \begin{array}{ll}
    -\text{div}\,\sigma (u)
    =f(x)
    & (x\in\Omega\setminus\Gamma),\\
    u=g(x) & (x\in\partial_D\Omega),\\
    \sigma(u)\nu=q(x) & (x\in\partial_N\Omega),\\
    \sigma(u_1)\nu=\zeta (x) [u]=\sigma(u_2)\nu\hspace{20pt}
    & (x\in \Gamma ).
  \end{array}
  \right.
\end{equation}
The meanings of the above symbols are as follows.
We use the Einstein summation convention in this section.
For matrices $\xi=(\xi_{kl}),~\eta=(\eta_{kl})\in \R^{d\times d}$,  
we denote their component-wise inner product by 
$\xi :\eta :=\xi_{kl}\eta_{kl}$ 
and the norm by $|\xi |:=\sqrt{\xi :\xi}$.

The solution $u$ is a displacement field on
$\Omega\setminus\Gamma =\Omega_1\cup\Omega_2$.
We denote $u|_{\Omega_i}$ by $u_i$ for $i=1,2$,
and often write $u=(u_1,u_2)$.
The symmetric gradient of
$u$ is defined by
\begin{equation*}
  e(u):=\frac{1}{2}\left(\nabla u^{\mathrm{T}}
  +\left(\nabla u^{\mathrm{T}}\right)^{\mathrm{T}}\right)
  \in \R^{d\times d}_{\mathrm{sym}}.
\end{equation*}
The stress tensor $\sigma(u)\in\R^{d\times d}_{\mathrm{sym}}$
satisfies the constitutive relation
\begin{equation*}
  \sigma (u):=\C e(u)=\left( c_{klmn}\,e_{mn}(u)\right)_{k,l} 
  \in \R^{d\times d}_{\mathrm{sym}},
\end{equation*}
where $\C:=\C(x)=(c_{klmn}(x))\in\R^{d\times d\times d\times d}$
is the elasticity tensor with the symmetries $c_{klmn}=c_{mnkl}=c_{lkmn},(1\leq k,l,m,n\leq d)$. 

We assume that $\C\in L^\infty(\Omega;\R^{d\times d\times d\times d})$ and
there exists $c_*>0$ such that
\begin{equation*}
c_{klmn}(x)\xi_{kl}\xi_{mn}\geq c_*|\xi|^2  
\quad
(\text{a.e.}~x\in \Omega,~\xi \in\R^{d\times d}_{\mathrm{sym}}).
\end{equation*}

The first equation of \eqref{model} is the force balance
equation in each subdomain $\Omega_i$, where $f$ is a given body force.
The second and third equations
are Dirichlet and Neumann boundary conditions, where
$g$ is a given displacement on $\partial_D\Omega$
and $q$ is a given surface traction on $\partial_N\Omega$.

In the fourth equation,
$\zeta(x)\ge 0$ is a given adhesive parameter on the
glued surface $\Gamma$ which
represents the strength of adhesive bonding.
The adhesive force at $x\in\Gamma$ is assumed to be
$\zeta(x)[u]$,
where $[u(x)]:=(u_2(x)-u_1(x))$ is the gap of the displacement
$u=(u_1,u_2)$ on $\Gamma$.
It should be balanced
with the surface traction force $\sigma(u_1(x))\nu$
on $\Omega_1$ side and also with 
$\sigma(u_2(x))\nu$
on $\Omega_2$ side.

To consider a weak formulation of \eqref{model},
we introduce the following spaces.
\begin{align*}
  &X_i:= H^1(\Omega_i;\R^d),\quad
  V_i:=\{u_i\in X_i;~u_i|_{\partial_D\Omega_i}=0\}
  \quad (i=1,2),\\
  &X:=H^1(\Omega\setminus\Gamma;\R^d)
  \cong X_1\times X_2,\quad
  V:=\{u\in X;~u|_{\partial_D\Omega}=0\}  \cong V_1\times V_2.
\end{align*}
For $g=(g_1,g_2)\in X$, we also define
affine spaces:
\begin{align*}
  V(g):=V+g=\{u\in X;~u-g\in V\},\quad
  V_i(g_i):=V_i+g_i\quad (i=1,2).
\end{align*}
\begin{Def}[Weak solution]
{\rm 
We suppose that a Dirichlet boundary data $g\in X$,
a body force $f\in L^2(\Omega;\R^d)$, a surface traction
$q\in L^2(\partial_N\Omega;\R^d)$, and an adhesive coefficient
$\zeta \in L^\infty (\Gamma)$, $\zeta (x)\ge 0$
are given.
Then, we call $u$ a weak solution of \eqref{model}
if
\[
u\in V(g),\quad
a_0(u,v)=l_0(v)~~\mbox{for all }v\in V,
\]
where
\begin{align}
  &a_0(u,v):=\int_{\Omega\setminus\Gamma}
  \sigma(u):e(v)\,dx
  +\int_\Gamma\zeta [u]\cdot[v]\,ds
  \quad (u,v\in X),\label{a_0}\\
&l_0(v):=\int_{\Omega\setminus\Gamma}f\cdot v\,dx+\int_{\partial_N\Omega}q\cdot v\,ds,
  \quad(v\in X).\notag
\end{align}
}
\end{Def}

The bilinear form $a_0$ and the linear form $l_0$ are decomposed into sum of the
following subforms:
\begin{align*}
&a_0(u,v)=\sum_{i=1}^2 a_i(u_i,v_i)+a_\Gamma ([u],[v])\quad (u,~v\in X),\\
  &a_i(u_i,v_i):=
  \int_{\Omega_i}\sigma(u_i):e(v_i)\,dx,~(i=1,2),
  \quad
  a_\Gamma(u,v):=\int_\Gamma\zeta u\cdot v\,ds.\\
  &l_0(v)=\sum_{i=1}^2l_i(v_i),
  \quad
  l_i(v_i):=\int_{\Omega_i}f\cdot v_i\,dx+\int_{\partial_N\Omega_i}q\cdot v_i\,ds,
  ~~(v_i\in X_i,~i=1,2)
\end{align*}

We remark that, if $c_{klmn}\in C^1(\overline{\Omega_i})$ for $i=1,2$, then
a strong solution of \eqref{model}, i.e.,
$u\in H^2(\Omega\setminus\Gamma;\R^d)$ which satisfies \eqref{model}
almost everywhere on $\Omega\setminus\Gamma$ or on $\partial\Omega\cup\Gamma$,
is a weak solution. On the other hand, if a weak solution belongs to
$H^2(\Omega\setminus\Gamma;\R^d)$, then it is a strong solution.

\section{Unique existence of a weak solution}
\label{sec:3}

For establishing the unique existence of a weak solution,
the coercivity of the bilinear form $a_0(u,v)$
defined in \eqref{a_0} is essential.
\begin{Th}[Coercivity of $a_0$]\label{th-coercivity}
We suppose that 
$\zeta(x) \ge 0$ and $\|\zeta \|_{L^\infty(\Gamma)}>0$.
Then there exists $a_*>0$ such that
$a_0(v,v)\ge a_*\|v\|_X^2$ holds for all $v\in V$.
\end{Th}
A slightly long proof of this theorem using an argument by contradiction
was given in \cite{Yoneda2018} and another simpler proof
will be given in our forthcoming paper.
We postpone the proof to it and here we just remark that
the case of $\H^{d-1}(\partial_D\Omega_i)>0$ for both $i=1,2$
is relatively easily shown by using K\"{o}rn's second inequality
\cite{D-L1976}.

As a consequence of Theorem~\ref{th-coercivity}, we immediately
have the unique existence of a weak solution.
\begin{Th}[Unique existence]\label{th-uni-ex}
We suppose that Dirichlet boundary data $g\in X$,  
a body force $f\in L^2(\Omega;\R^d)$, a surface traction
$q\in L^2(\partial_N\Omega;\R^d)$
are given.
Then, under the conditions of Theorem~\ref{th-coercivity},
there exists a unique weak solution to \eqref{model}.
\end{Th}
\begin{proof}
\smartqed
Under the conditions, from the definitions of $a_0$ and $l_0$,
we can show that 
$a_0$ is a continuous symmetric bilinear form on
$X\times X$ and $l_0$ is a continuous linear form on $X$.

We set $\tilde{u}:=u-g$. Then $u$ is a weak solution to \eqref{model}
if and only if
\begin{align}\label{tildeu}
\tilde{u}\in V,\quad a_0(\tilde{u},v)=l_0(v)+a_0(g,v)\quad (v\in V).
\end{align}
From the Lax-Milgram lemma \cite{Ciarlet1978} and Theorem~\ref{th-coercivity},
there exists a unique $\tilde{u}$ which satisfies \eqref{tildeu}.
Hence the unique existence of the weak solution has been proved.
\qed
\end{proof}

A variational principle for the above symmetric Lax-Milgram type problem
is also well known. The weak solution $u^*\in V(g)$ is a unique minimizer of the
following energy:
\begin{align}
u^*=\argmin_{u\in V(g)}E(u),\label{argmin}
\end{align}
where
\begin{align}\label{E(u)}
E(u):=\frac{1}{2}a_0(u,u)-l_0(u).
\end{align}

In Section~\ref{sec:5}, we consider a finite element approximation
for our model \eqref{model}. So called C\'ea's lemma \cite{Ciarlet1978}
implies the following error estimate.
We define
\[
a^*:=\sup_{v,w\in V}\frac{a_0(v,w)}{\|v\|_X\|w\|_X}<\infty .
\]
\begin{Prop}\label{cea}
Under the assumptions of Theorem~\ref{th-uni-ex},
we suppose that $V_h$ is a closed subspace of $V$.
Then there uniquely exists $u_h$
such that
\begin{align}\label{uh}
u_h\in V_h(g):=V_h+g,\quad a_0(u_h,v_h)=l_0(v_h)\quad (v_h\in V_h).
\end{align}
Furthermore, it satisfies
\[
\| u-u_h\|_X\le \frac{a^*}{a_*}\inf_{v_h\in V_h(g)} \|u-v_h\|_X.
\]
\end{Prop}
The problem \eqref{uh} corresponds to the finite element scheme.
If $V_h$ is a space of 
piece-wise linear element (P1 element) on a regular triangular mesh,
it is known that $\inf_{v_h\in V_h(g)} \|u-v_h\|_X=O(h)$ as
the mesh size $h$ tends to $0$ under suitable regularity for
$u$ and the triangular mesh \cite{Ciarlet1978}.

 \section{Alternating iteration method}   
\label{sec:4}

We remark that $u^*=(u_1^*,u_2^*)\in V(g)$ is a weak solution to \eqref{model}
if and only if 
\begin{align}
  &a_1(u_1^*,v_1)+a_\Gamma (u_1^*,v_1)=a_\Gamma (u_2^*,v_1)+l_1(v_1)
  \quad (^\forall v_1\in V_1),\label{u*1}\\
  &a_2(u_2^*,v_2)+a_\Gamma (u_2^*,v_2)=a_\Gamma (u_1^*,v_2)+l_2(v_2)
  \quad (^\forall v_2\in V_2).\label{u*2}
\end{align}

We consider the following alternating method. 

\noindent
{\bf Gauss-Seidel type scheme}\\    
For given $u_2^0\in V_2(g)$, and  
for $m=0,1,2,\cdots$, find $u^m=(u_1^m,u_2^m)\in V(g)$ such that
\begin{align}
  &a_1(u_1^m,v_1)+a_\Gamma (u_1^m,v_1)=a_\Gamma (u_2^{m-1},v_1)+l_1(v_1)
  \quad (^\forall v_1\in V_1,~m=1,2,\cdots ),\label{uGS1}\\
  &a_2(u_2^m,v_2)+a_\Gamma (u_2^m,v_2)=a_\Gamma (u_1^m,v_2)+l_2(v_2)
  \quad (^\forall v_2\in V_2,~m=1,2,\cdots ).\label{uGS2}
\end{align}

We call the above alternating iteration method ``Gauss-Seidel type''
by analogy with an iterative solver for linear systems. 
The unique solvability of each $u_i^m$ is clear from K\"{o}rn's second
inequality.

The following theorem tells us that the Gauss-Seidel type scheme 
defines a sequence $\{u^m\}$ which monotonically
decreases the total energy $E(u)$ defined in \eqref{E(u)}.
\begin{Th}\label{GS-energy-decay}
  The obtained sequence $\{u^m=(u_1^m,u_2^m)\}_{m}$ by the Gauss-Seidel type scheme  
  satisfies the following energy decay property:
  \begin{align}\label{EEE}
  E((u_1^{m-1},u_2^{m-1}))\ge   E((u_1^m,u_2^{m-1}))\ge E((u_1^m,u_2^m)) 
  \quad
  (m=1,2,\cdots).
  \end{align}
\end{Th}
\begin{proof}
  For simplicity, we denote $(u_1^m,u_2^n)\in V(g)$ by $u^{m,n}$.
The second inequality is shown as follows.
Setting $v=(v_1,v_2):=u^{m,m-1}-u^{m,m}=(0,u_2^{m-1}-u_2^m)$, we have
\begin{align*}
  &\ E(u^{m,m-1})-E(u^{m,m})\\
  =&\
  \frac{1}{2}a_0(u^{m,m-1}+u^{m,m},u^{m,m-1}-u^{m,m})
  -l_0(u^{m,m-1}-u^{m,m})\\
  =&\ 
  \frac{1}{2}a_0(v,v)+a_0(u^{m,m}, v)-l_0(v)\\ 
  \ge &\
  a_0(u^{m,m}, v)-l_0(v)\\ 
  =&\
  a_1(u_1^m, v_1)+a_2(u_2^m,v_2)+a_\Gamma (u_2^m-u_1^m,v_2-v_1) -l_2(v_2)\\  
  =&\ 0,
\end{align*}
where we have used $v_1=0$ 
and \eqref{uGS2} for the last equality. 
The first inequality in \eqref{EEE} is shown in the same way.
\qed
\end{proof}
Since the weak solution $u^*$ is the minimizer of the total energy $E$
as written in \eqref{argmin}, the sequence generated by the Gauss-Seidel type
scheme is expected to approximate $u^*$. We will study it numerically in
the next section. 
\section{Numerical solution}\label{sec:5}
First we recall the assumption $\H^{d-1}(\partial_D\Omega)>0$.
In this section we only deal with the case 
$\partial_D\Omega\subset\partial\Omega_{1}$ and $\zeta(x)>0$
on $\Gamma$. 

\subsection{A matrix representation of the monolithic formulation}
We briefly review a matrix representation of the monolithic formulation. %
Let $\Lambda^{(i)}$ be an index for finite element basis function and 
be decomposed as $\Lambda^{(i)}_{I}\cup\Lambda^{(i)}_{B}$, corresponding to nodes in the subdomain
$
\overline{\Omega_{i}}
\setminus\Gamma$ and on the common boundary $\Gamma$. 
\par
We define the following stiffness matrices in each subdomain
$\Omega_i$ 
using the bilinear forms $\{a_{i}(\cdot, \cdot)\}_{i=1,2}$ defined in Section \ref{sec:2}.
\begin{align*}
  [A^{(i)}_{\mu\,\nu}]_{k\,l}&=a_{i}(\varphi_{l}^{(i)},\varphi_{k}^{(i)}) &
  k\in \Lambda_{\mu}^{(i)}, 
  l\in\Lambda_{\nu}^{(i)}, \{\mu,\nu\}\in\{I, B\},\\
  [M^{(i\,j)}]_{k\,l}&=a_{\Gamma}(\varphi_{l}^{(j)}, \varphi_{k}^{(i)}) & i,j\in\{1,2\},\ k\in\Lambda_{B}^{(i)},l\in\Lambda_{B}^{(j)}.
\end{align*}
Here combination of four mass matrices $\{M^{(i\,j)}\}$ provides a matrix representation of the bilinear form
$a_\Gamma ([\cdot ], [\cdot ])$. 
A matrix representation of the monolithic formulation reads
\begin{equation}\label{eqn:mono-matrix}
  \begin{bmatrix}
    A_{II}^{(1)} &     A_{IB}^{(1)} &  & \\
    A_{BI}^{(1)} &   \mbox{}\hspace{2mm}  A_{BB}^{(1)}+M^{(11)}\hspace{2mm}\mbox{} &  -M^{(12)} \\
    & -M^{(21)} & A_{BB}^{(2)}+M^{(22)} \hspace{2mm}\mbox{}&   A_{BI}^{(2)} \\
    & & A_{IB}^{(2)} & A_{II}^{(2)} 
  \end{bmatrix}
  \begin{bmatrix}
    u_{1,I} \\ u_{1,B} \\ u_{2,B} \\ u_{2,I}
  \end{bmatrix}
=
  \begin{bmatrix}
    f_{1,I} \\ f_{1,B} \\ f_{2,B} \\ f_{2,I}
  \end{bmatrix},
\end{equation}
where the right hand side consists of the body force $f$ 
and inhomogenous
Dirichlet and Neumann data $g$ and $q$. 
\begin{Rem}
  The monolithic method can be computed by $LDU$-factorization with any symmetric permutation because the stiffness matrix is positive definite thanks to the coercivity of the weak form with
  $\H^{d-1}(\partial_D\Omega)>0$
  (Theorem~\ref{th-coercivity}, Proposition~\ref{cea}). 
\end{Rem}
\subsection{Alternating iterative
method in discrete form  
}
In the following, we suppose that
the nodal points and the surface meshes of the mesh decomposition of 
domains $\Omega_{1}$ and 
$\Omega_{2}$ coincide on the interface $\Gamma$, 
which leads to
\begin{equation*}
  M=M^{(11)}=M^{(12)}=M^{(21)}=M^{(22)}\,.
\end{equation*}
Let us define an inner product of vector $u_{i,\,\mu}$ and $v_{i,\,\mu}$ as  
$(u_{i,\,\mu},v_{i,\,\mu}):=\sum_{k\in\Lambda_{\mu}^{(i)}}(u_{\mu})_{k}(v_{\mu})_{k}$ for $i=1,2$
and $\mu\in\{I,B\}$, 
and denote the standard $\ell^{2}$-norm by
$\|u_{i,\,\mu}\|=(u_{i,\,\mu},u_{i,\,\mu})^{1/2}$. 
  Since $\zeta(x)>0$ on $\Gamma$,
  the mass matrix $M$ is positive definite
  and $(M u_{i,B}, v_{i,B})$ becomes an inner product with weight $M$.
Hence, we denote a norm with the weight $M$ by $\|u_{i,B}\|_{M}=(M\,u_{i,B}, u_{i,B})^{\frac{1}{2}}$. 
Then there exist $\beta_{1}>0$ and $\beta_{2}>0$ such that 
$\beta_{1}\|u_{i,B}\|_{M}^{2}\leq \|u_{i,B}\|^{2}\leq \beta_{2}\|u_{i,B}\|_{M}^{2}$ 
holds for any $u_{i,B}$. 

We prepare two matrices in each subdomain to describe the linear system in a simpler way,
\begin{equation*}
  {\cal A}^{(i)}:=  \begin{bmatrix}
    A_{II}^{(i)} & A_{IB}^{(i)} \\    A_{BI}^{(1)} & A_{BB}^{(i)} 
  \end{bmatrix},\quad
  \widetilde{\cal A}^{(i)}:=  \begin{bmatrix}
    A_{II}^{(i)} & A_{IB}^{(i)} \\    A_{BI}^{(i)} & A_{BB}^{(i)} + M
  \end{bmatrix}\,.
\end{equation*}
  \begin{Lem}\label{lem:corecivity-omega1}
    There exists $\alpha_1>0$ such that
\begin{equation*}
\left(
{\cal A}^{(1)}
  [v_{1,I}\ \ v_{1,B}]^{T},
  [v_{1,I}\ \ v_{1,B}]^{T} \right)\geq{}\alpha_{1}(\|v_{1,I}\|^{2}+\|v_{1,B}\|^{2})
\geq{}\alpha_{1} \beta_{1}\|v_{1,B}\|_{M}^{2},
\end{equation*}
holds for any vector $[v_{1,I}\ \  v_{1,B}]^{T}$.
  \end{Lem}
  \begin{proof}\smartqed
    Since $  {\cal A}^{(1)}$ in $\Omega_{1}$ is positive definite due to
    the Dirichlet boundary $\partial_{D}\Omega\subset\partial\Omega_{1}$,
    there exists $\alpha_1 >0$ and the first inequality holds.
    The second one is clear from the definition of $\beta_1$.
    \qed
  \end{proof}

The Gauss-Seidel type iteration defined in $(\ref{uGS1})$ and $(\ref{uGS2})$ 
is written in
the following 
matrix presentation.
\begin{Alg}[Gauss-Seidel type iteration]{\rm
    Let $u_{2,B}^{0}$ be an initial guess
of $u_2$ 
    on $\Gamma$. From obtained data $u_{2,B}^{m-1}$ of $m-1$-step, approximate solution
    $[u_{1,I}^{m}\ u_{1,B}^{m}]^T$ and $[u_{2,I}^{m}\ u_{2,B}^{m}]^T$ are
generated by solving following two problems successively,
\begin{equation*}\label{eqn:GS-matrix}
\widetilde{{\cal A}}^{(1)}
\begin{bmatrix}u_{1,I}^{m} \\ u_{1,B}^{m}\end{bmatrix}
=\begin{bmatrix}f_{1,I} \\ f_{1,B}+M\,u_{2,B}^{m-1}\end{bmatrix}\ \text{ and then }\ %
\widetilde{{\cal A}}^{(2)}
\begin{bmatrix}u_{2,I}^{m} \\ u_{2,B}^{m}\end{bmatrix}
=\begin{bmatrix}f_{2,I} \\ f_{2,B}+M\,u_{1,B}^{m}\end{bmatrix}\,.
\end{equation*}
}
\end{Alg}
Let  $[u_{1,I}^{m}\ u_{1,B}^{m}]^T$ and $[u_{2,I}^{m}\ u_{2,B}^{m}]^T$
($m=1,2,\cdots$) 
be the solution of the Gauss-Seidel type iteration
and let $[u_{1,I}^{*}\ u_{1,B}^{*}]^T$ and $[u_{2,I}^{*}\ u_{2,B}^{*}]^T$
be the one of the monolithic system \eqref{eqn:mono-matrix}. 
We define the error between them by
\begin{align}\label{err}
  [e_{i,I}^{m}\ \ e_{i,B}^{m}]^{T}:=
  [u_{i,I}^{m}\ \ u_{i,B}^{m}]^{T}-[u_{i,I}^{*}\ \ u_{i,B}^{*}]^{T}\quad (i=1,2)
\end{align}
We have the following convergence estimates for the
Gauss-Seidel type iteration.
\begin{Lem}\label{e-lem}
  The error on the boundary admits the following estimates:
\begin{equation*}
  \|e_{1,B}^{m}\|_{M}\leq r\,\|e_{2,B}^{m-1}\|_{M}\,,
  \quad  \|e_{2,B}^{m}\|_{M}\leq \|e_{1,B}^{m}\|_{M}
  \quad (m=1,2,\cdots),
\end{equation*}
where $r:=1/(1+\alpha_{1}\beta_{1})<1$.
\end{Lem}
\begin{proof}
  From the definition of the errors \eqref{err},
they satisfy the following linear systems:
\begin{equation}\label{Ae}
\widetilde{\cal A}^{(1)}
 [e_{1,I}^{m}\ \ e_{1,B}^{m}]^{T}=
 [0 \ \ M\,e_{2,B}^{m-1}]^{T},
 \quad
\widetilde{\cal A}^{(2)}
 [e_{2,I}^{m}\ \ e_{2,B}^{m}]^{T}=
  [0 \ \ M\,e_{1,B}^{m}]^{T}\,.
\end{equation}
Taking inner product of $[e_{1,I}^{m}\ \ e_{1,B}^{m}]^{T}$ 
with the left equation of \eqref{Ae},  and of $[e_{2,I}^{m}\ \ e_{2,B}^{m}]^{T}$ with the right, we obtain 
\begin{align}
&\left({\cal A}^{(1)}
[e_{1,I}^{m}\ \ e_{1,B}^{m}]^{T}, [e_{1,I}^{m}\ \ e_{1,B}^{m}]^{T}\right)
+\| e_{1,B}^{m}\|_M^2=(M\,e_{2,B}^{m-1},e_{1,B}^{m})
\le \| e_{2,B}^{m-1}\|_M \| e_{1,B}^{m}\|_M\,,\label{Aee1}\\ 
&\left({\cal A}^{(2)}
[e_{2,I}^{m}\ \ e_{2,B}^{m}]^{T}, [e_{2,I}^{m}\ \ e_{2,B}^{m}]^{T}\right)
+\| e_{2,B}^{m}\|_M^2=(M\,e_{1,B}^{m},e_{2,B}^{m})  
\le \| e_{1,B}^{m}\|_M \| e_{2,B}^{m}\|_M\,.\label{Aee2}
\end{align}
From Lemma~\ref{lem:corecivity-omega1} and \eqref{Aee1}, we have
\begin{equation*}
  (1+\alpha_{1}\beta_{1})\|e_{1,B}^{m}\|_{M}^{2}
\leq\|e_{2,B}^{m-1}\|_{M}\|e_{1,B}^{m}\|_{M}\,.
\end{equation*}
This gives the first inequality. The second inequality is obtained from
\eqref{Aee2} with positive semi-definiteness of ${\cal A}^{(2)}$. 
\smartqed \qed
\end{proof}
\begin{Th}
There exists $C>0$ such that the following error estimate holds:
\begin{equation*}
  \sqrt{\|e_{i,I}^{m}\|^{2}+\|e_{i,B}^{m}\|^{2}}\le Cr^{m+i-2}  \|e_{2,B}^{0}\|_{M} 
  \quad (i=1,2,~m=1,2,\cdots).
\end{equation*}
where $r:=1/(1+\alpha_{1}\beta_{1})<1$.
\end{Th}
\begin{proof}
  From Lemma~\ref{e-lem}, we have
\begin{equation*}
\|e_{2,B}^{m}\|_{M}\leq \|e_{1,B}^{m}\|_{M}\leq r\,\|e_{2,B}^{m-1}\|_{M}\,
  \quad (m=1,2,\cdots). 
\end{equation*}  
  These inequalities imply
  \begin{align}\label{erm}
    \|e_{i,B}^{m}\|_{M}\leq r^m\,\|e_{2,B}^{0}\|_{M}
    \quad (i=1,2,~m=1,2,\cdots).
  \end{align}
  Since the matrices $\widetilde{\cal A}^{(i)}$ $(i=1,2)$ are invertible,
  from \eqref{Ae}, we obtain
  \begin{align*}
    &\sqrt{ \| e_{1,I}^{m}\|^2+\| e_{1,B}^{m}\|^2}\le
    \left\| (\widetilde{\cal A}^{(1)})^{-1}\right\|\, \| Me_{2,B}^{m-1}\|
   =  \left\| (\widetilde{\cal A}^{(1)})^{-1}\right\|\, \|M^{1/2}e_{2,B}^{m-1}\|_{M}\,,\\ 
    &\sqrt{ \| e_{2,I}^{m}\|^2+\| e_{2,B}^{m}\|^2}\le
    \left\| (\widetilde{\cal A}^{(2)})^{-1}\right\|\, \| Me_{1,B}^{m}\|
    =  \left\| (\widetilde{\cal A}^{(2)})^{-1}\right\|\, \|M^{1/2}e_{1,B}^{m}\|_{M}\,. \\ 
  \end{align*}
  Together with the estimate \eqref{erm} and with the fact that $M$ is positive definite, there exists a $C>0$ such that
the assertion of the theorem holds.
  \smartqed \qed
\end{proof}
  
\begin{Rem}
The Gauss-Seidel type iteration is straightforwardly extended to SOR type iteration by introducing 
a relaxation parameter.
\end{Rem}

\subsection{Numerical results}
Now we
show numerical verification on convergence of Gauss-Seidel type iteration using a manufactured solution, 
\begin{equation*}
  \begin{bmatrix}
  u_{1}\\ u_{2}\\ u_{3}
  \end{bmatrix}
= \begin{bmatrix}
\sin((\pi/8)\, x_{1})\times\cos((\pi/8)\, x_{2})\times\sin((\pi/16)\, x_{3})\\
\cos((\pi/16)\, x_{1})\times\sin((\pi/8)\, x_{2})\times\cos((\pi/8)\, x_{3})\\
\cos((\pi/8)\, x_{1})\times\sin((\pi/16)\, x_{2})\times\sin((\pi/8)\, x_{3})
\end{bmatrix}
\end{equation*}
in $\Omega_{1}=(0,4)\times(0,2)\times (0,4)$,
$\Omega_{2}=(0,4)\times(2, 4)\times (0,4)$, and
$\partial_{D}\Omega=\{(x,y,z)\,;\,0<x<4, y=0, 0<z<4\}$
with corresponding inhomogeneous Dirichlet
data $g(x)$,
the load $f(x)$ on $\Omega$, and the surface traction $q(x)$ on $\partial_N\Omega$. 
For simplicity, we put $\zeta(x)\equiv 1$.
Figure \ref{fig:1} shows relative errors of finite element solution discretized with P1 
element solved by the monolithic formulation, which 
ensures the 1st order approximation error of the solution, $O(h)$ with mesh size $h$.
\begin{figure}[b]
  \begin{minipage}{\textwidth}
    \begin{center}
      \includegraphics[width=0.5\textwidth]{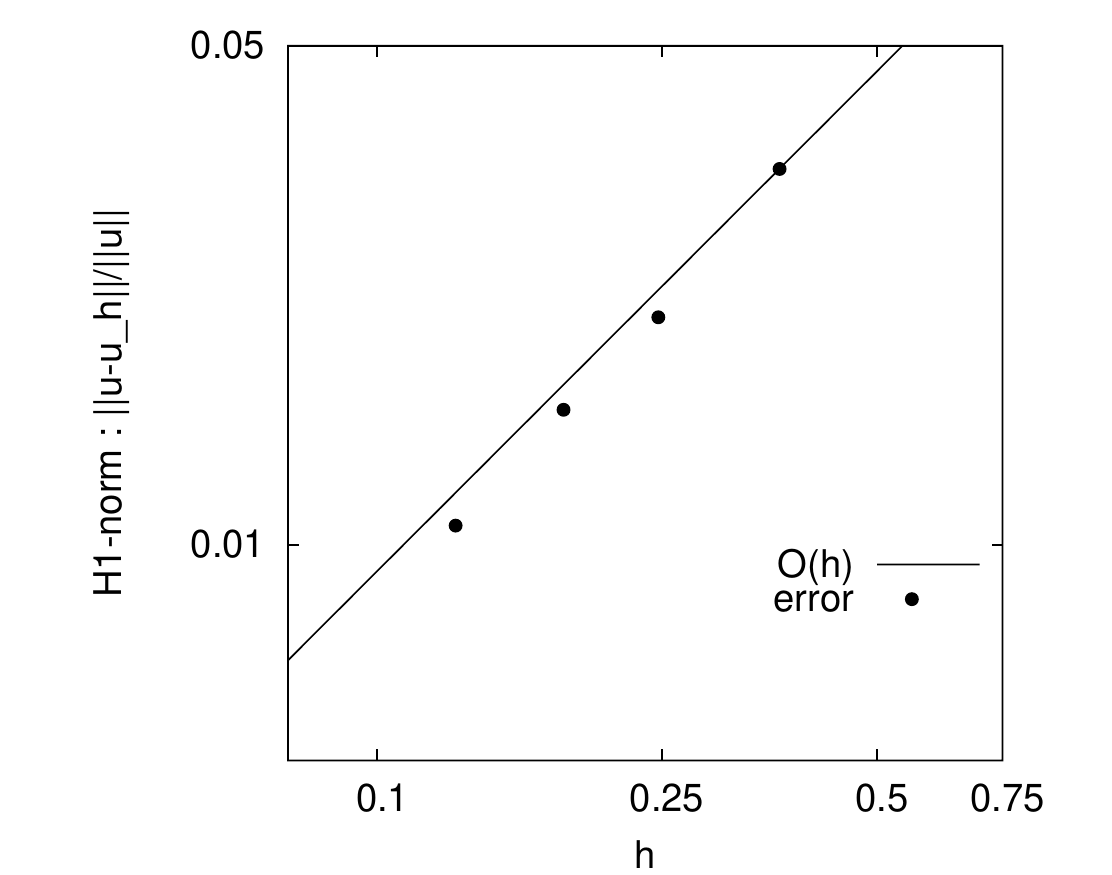}
    \end{center}
    \caption{Relative errors of P1 finite element solution computed by monolithic formulation}
\label{fig:1}
  \end{minipage}
\end{figure}
Convergence of Gauss-Seidel type 
iteration to the solution by the monolithic formulation with the relative error measured by $\|\cdot\|_{H^{1}(\Omega)}$ is shown in the left of Figure \ref{fig:2}, and
relative error to the manufactured solution in the right of Figure \ref{fig:2}. 
Here mesh subdivisions 
with $h_{\max}= 0.36551$ in $20\times 20 \times 20$, and
 $h_{\max}=0.12878 $ in $60\times 60 \times 60$ are used. %
We can see the convergence 
does not depend on the mesh size,
and the same relative error as one by monolithic formulation to the manufactured solution is obtained after certain iterations, though Gauss-Seidel type iteration continues to converge.  
\begin{figure}[tb]
  \begin{minipage}{\textwidth}
    \begin{center}
      \includegraphics[width=0.45\textwidth]{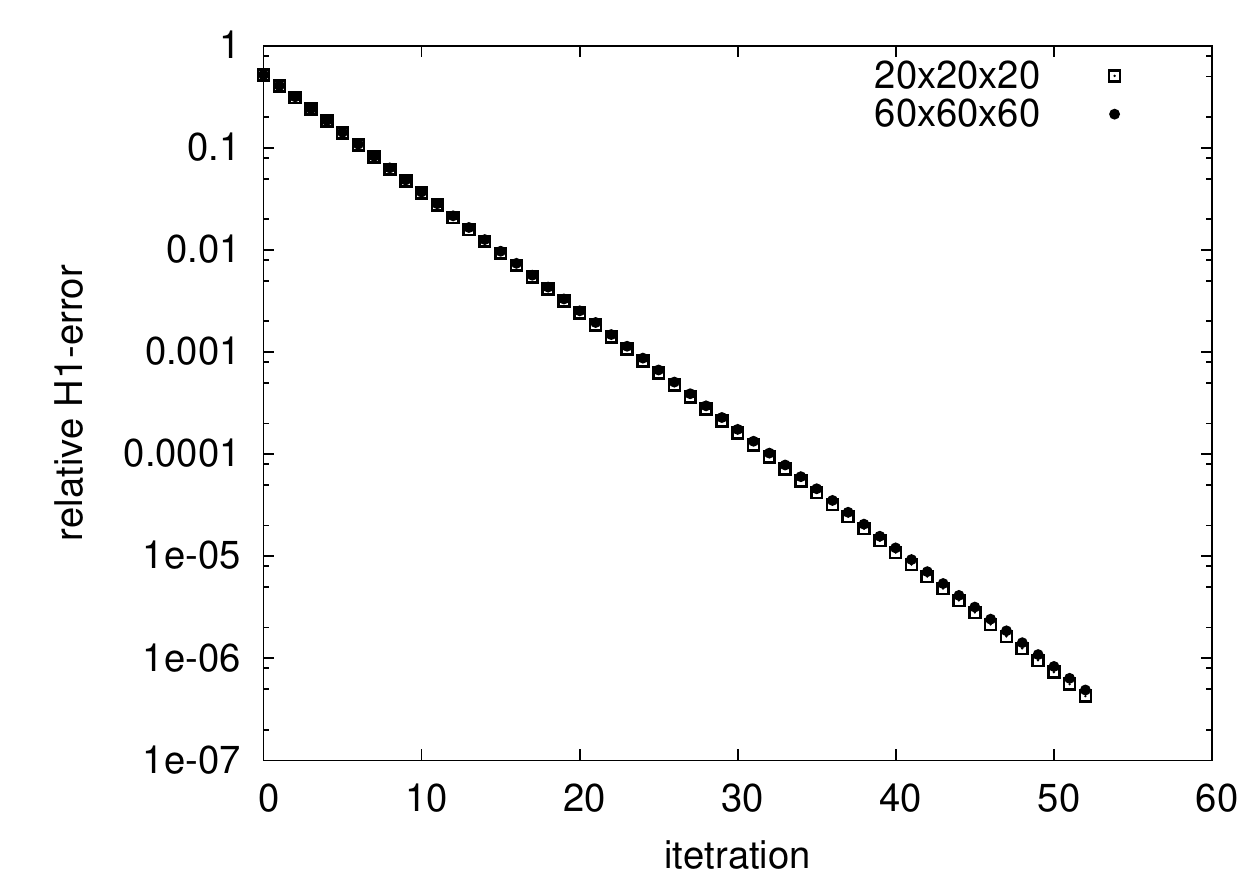} \ %
      \includegraphics[width=0.45\textwidth]{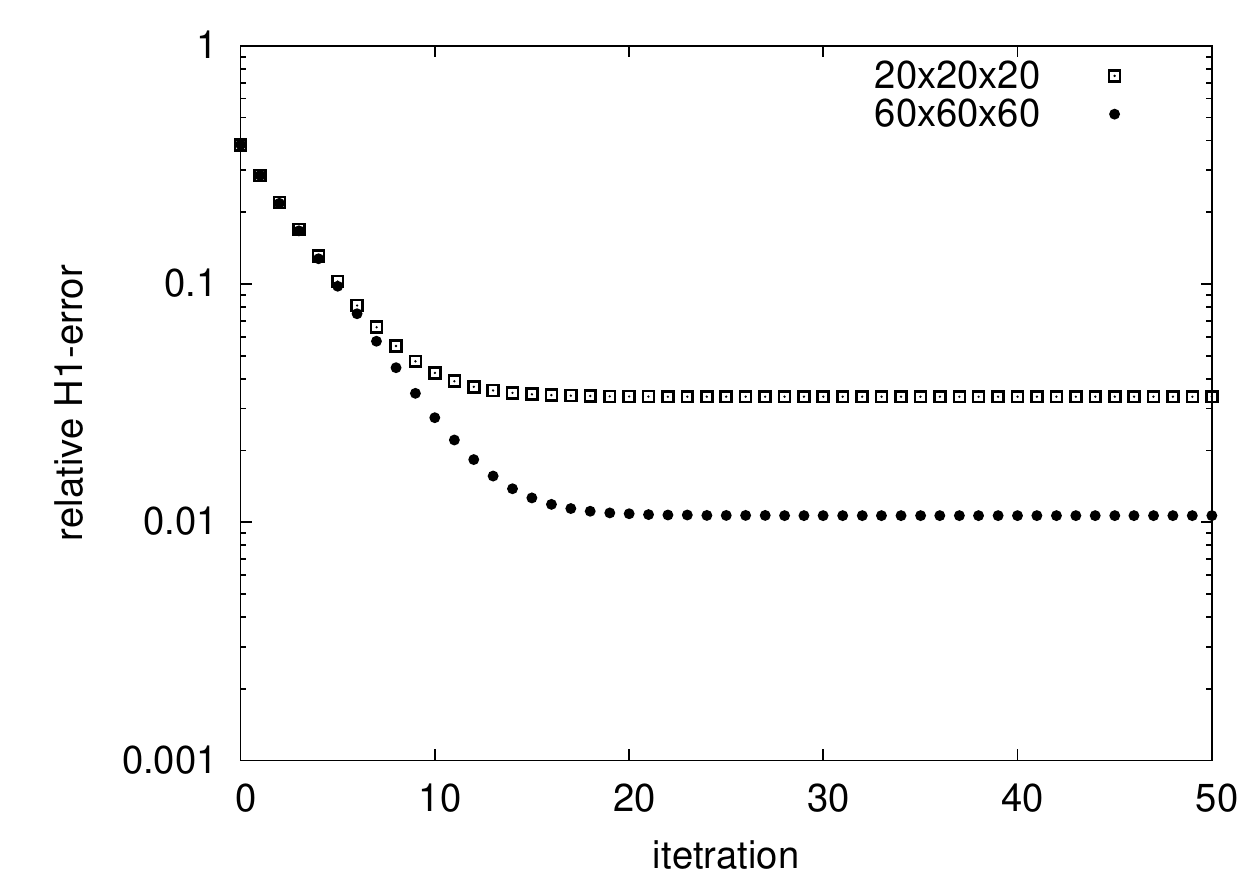}
    \end{center}
    \caption{Convergence history of Gauss-Seidel type iteration}
\label{fig:2}
  \end{minipage}
\end{figure}
\subsection{Computational efficiency}
We used FreeFem++ software package and Dissection sparse direct solver~\cite{SuzukiRoux2014} to obtain finite element solution. %
Table \ref{tab:1} shows computational time of the direct solver for the monolithic formulation and the Gauss-Seidel type iteration, using Intel Core processor i7-6770HQ with 4 cores running at 2.60GHz. %
$LDU$-factorization is performed before starting iteration and forward/backward substitutions are repeated because the stiffness matrix in each subdomain does not change during the iteration. %
Since computational complexity of $LDU$-factorization of sparse matrix 
with P1 finite element is more than $O(N^{2})$ with number of unknowns $N$, 
factorization cost for sub-matrices in Gauss-Seidel type is less than half, $2\times (N/2)^{2}/N^{2}=1/2$ of the one for monolithic formulation.
We can observe shorter CPU time for factorization of Gauss-Seidel type, which is shown as number in parentheses, when the problem size is enough large.
In Dissection solver, numerical factorization is fully parallelized but
there exist some sequential processing parts, e.g. fill-in analysis of the sparse matrix, which masks speed-up of elapsed time for the factorization in Gauss-Seidel type.
We also observe that elapsed time of iteration in Gauss-Seidel type with selected iterating number to obtain appropriately approximate solution is less than time
of the factorization. Hence if we can find a reasonable criteria to stop iteration, the Gauss-Seidel type iteration becomes more efficient than monolithic formulation.
\begin{table}[htb]
  \begin{minipage}{\textwidth}
    \begin{center}
      \caption{Elapsed time of monolithic and Gauss-Seidel methods by FreeFem++ and Dissection solver with 4 cores. CPU time for factorization are also shown within parentheses.}
      \begin{tabular}{ccccrrrrr}
 mesh & $h_{\max}$ & solver & \# unknowns &\mbox{}\hspace{-3mm}\# iteration\hspace{-4mm}\mbox{} & \multicolumn{4}{c}{time (sec.)}\\
& & &  & & \multicolumn{2}{c}{factorization} & iteration & total \\
        \hline
       $20\times 20\times 20$ & 0.36551 & monolithic & 30,870& --- & 1.364 & ( 3.147 )& --- \ \ \mbox{}& 1.364 \\
       &  & Gauss-Seidel & 14,742+16,128 &12 & 1.339 &( 2.752 ) & 0.245 & 1.585 \\
         \hline
       $30\times 30\times 30$ & 0.24729  &  monolithic & 97,743 &--- &  6.901 & ( 18.588 ) & --- \ \ \mbox{}& 6.901\\
      &  & Gauss-Seidel & 46,965+50,778 & 15 & 5.082 & ( 12.425 )& 1.286 & 6.368\\

        \hline
        $60\times 60\times 60$ & 0.12783  & monolithic & 748,470 & --- & 174.123& ( 603.567 ) & --- \ \ \mbox{}& 174.123 \\
       & & Gauss-Seidel & 368,928+379,542 & 20 & 118.973 & ( 316.482 ) & 25.559 &  144.532\\
        \hline
      \end{tabular}
    \end{center}
\label{tab:1}
\end{minipage}
\end{table}

\section{Conclusion}
We considered a stationary deformation problem for glued elastic bodies
and have established solvability of
its weak formulation.
We proposed a kind of alternating iteration scheme to approximate 
the problem and showed that the scheme has a nature of  
alternating minimizing algorithm with respect to the total energy. 

We proved the convergence for the Gauss-Seidel type iteration
in a rate of $O(r^m)$ with $r\in (0,1)$ in discrete setting.
Computational efficiency of the monolithic and   
alternating iterative algorithms have been verified with a three dimensional
problem.
The
alternating iterative method requires smaller computational resource
than the monolithic method
and also it has an advantage in computation time
when the degree of freedom is sufficiently large.

\begin{acknowledgement}
  The authors are grateful to Prof. Frederic Hecht for his useful comments on
  numerical computation with FreeFem++.
This work is supported by JSPS KAKENHI Grant Number JP16H03946 and JP17H02857.
\end{acknowledgement}

\end{document}